\documentclass[letterpaper, 11pt]{amsart}

\usepackage{geometry}
\usepackage[utf8]{inputenc}
\usepackage[T1]{fontenc}      
\usepackage{textgreek} 
\usepackage{amsmath}
\usepackage{amssymb}
\usepackage{amsthm}
\usepackage{mathtools}
\usepackage{mathrsfs}
\usepackage{leftindex}
\usepackage{tikz-cd}
\usepackage{fullpage}
\usepackage{microtype}
\usepackage{colonequals}
\usepackage[
  backend=biber,
  style=alphabetic,
  date=year,
  url=false,
  doi=false,
  isbn=false,
  maxbibnames=99,
  maxalphanames=4,
  useprefix=true
]{biblatex}

\AtEveryBibitem{%
  \clearfield{isbn}

  \clearlist{location}
  \clearfield{address}

  \iffieldundef{journaltitle}
    {}
    {%
      \clearfield{url}%
      \clearfield{eprint}%
      \clearfield{archivePrefix}%
      \clearfield{primaryClass}%
    }%
}

\DeclareMathOperator{\IC}{IC}

\DeclareMathOperator{\Aut}{Aut}
\DeclareMathOperator{\For}{For}

\DeclareMathOperator{\Pic}{Pic}
\DeclareMathOperator{\codim}{codim}
\DeclareMathOperator{\gr}{gr}
\DeclareMathOperator{\occ}{occ}
\DeclareMathOperator{\amp}{amp}
\DeclareMathOperator{\EdgeOp}{Edge}
\DeclareMathOperator{\VertOp}{Vert}

\DeclareMathOperator{\QCoh}{QCoh}
\DeclareMathOperator{\DefOp}{Def}
\DeclareMathOperator{\IH}{IH}

\usepackage{amsthm} 
\usepackage[colorlinks=true, linkcolor=red, citecolor=blue]{hyperref}
\theoremstyle{plain}
\newtheorem{theorem}{Theorem}[section]
\newtheorem{lemma}[theorem]{Lemma}
\newtheorem{corollary}[theorem]{Corollary}
\newtheorem{proposition}[theorem]{Proposition}

\theoremstyle{definition}
\newtheorem{definition}[theorem]{Definition}

\newtheorem{question}{Question}

\theoremstyle{remark}
\newtheorem{remark}[theorem]{Remark}

\title{Support theorem of universal compactified Jacobians}
\author{Yifan Wu}
\address{Columbia University}
\email{yw4413@columbia.edu}
\date{\today}

\begin{document}

\begin{abstract}
    We prove a full support theorem for the relative good moduli space of the universal compactified Jacobian $\bar{\pi}\colon \overline{J}_{g,n}^{d,\phi}\to \overline{\mathcal{M}}_{g,n}$, showing that every direct summand appearing in the BBDG decomposition of $\mathrm{R}\bar{\pi}_*\IC(\overline{J}_{g,n}^{d,\phi})$ has full support on the base $\overline{\mathcal{M}}_{g,n}$. Moreover, we explicitly describe this decomposition governed by the derived pushforward of the constant sheaf on the universal curve. 
    
    The first proof synthesizes Maulik and Shen's generalization of Ng\^{o}'s support theorem, a decomposition theorem for the good moduli space morphism, and equivariant perverse sheaves. We also provide an independent second proof by variation of stability conditions and the support theorem for relative Jacobians by Migliorini, Shende, and Viviani.
\end{abstract}

\maketitle

\section{Introduction}
\subsection{Motivation}
The Jacobian variety associated with a nonsingular projective curve is a fundamental geometric object. Extending this over the boundary of the moduli space, Kass--Pagani \cite{kassExtensionsUniversalTheta2017, kassStabilitySpaceCompactified2019} and Melo \cite{meloUniversalCompactifiedJacobians2020} constructed the universal compactified Jacobian stack $\overline{\mathcal{J}}_{g,n}^{d,\phi}$ over the moduli stack of stable marked curves $\overline{\mathcal{M}}_{g,n}$. 

A striking feature of these compactified Jacobians is the invariance of their cohomology groups under the choice of degree $d$ and any non-degenerate stability condition $\phi$. This was proven by Migliorini--Shende--Viviani \cite{miglioriniSupportTheoremHilbert2021}, with an independent proof recently provided by Pandharipande--Petersen--Schmitt--Wood \cite{pandharipandeUniversalCompactifiedJacobians2026}. Nevertheless, the geometry of these spaces remains poorly understood when the stability condition $\phi$ is degenerate, as the presence of strictly semistable sheaves introduces severe singularities over the boundary strata.

On the other hand, a powerful suite of derived category techniques has emerged from Ng\^{o}'s proof of the Fundamental Lemma \cite{ngoLemmeFondamentalPour2010}, centered on the notion of weak abelian fibrations and its support theorem. Subsequently, Chaudouard--Laumon \cite{chaudouardTheoremeSupportPour2016} and Maulik--Shen \cite{maulikCohomologicalHindependenceModuli2023} applied this theorem to deduce the full support property of the Hitchin fibration in the case of $G=GL_n$. In the present work, we adapt this methodology from the Hitchin setting to the universal compactified Jacobian.

Directly applying this machinery over the moduli of curves is obstructed by boundary complexity: even after a $\mathbb{G}_m$-rigidification, the universal compactified Jacobian stack may be non-proper over $\overline{\mathcal{M}}_{g,n}$. If we pass to the absolute good moduli space $J$, the base is replaced by the coarse space $\overline{M}_{g,n}$; the fiber over $[C]$ is the quotient $\bar{J}(C)/\Aut(C)$. This quotient prevents the natural translation action of the relative Picard group scheme from descending to $J$, and thus blocks a direct application of Ng\^{o}'s support theorem, which fundamentally uses a group scheme action.

To overcome this, we work with the relative good moduli space $\overline{J}_{g,n}^{d,\phi}$, which rigidifies vertical automorphisms in the fibers of $\overline{\mathcal{J}}_{g,n}^{d,\phi}\to\overline{\mathcal{M}}_{g,n}$ while keeping the base still $\overline{\mathcal{M}}_{g,n}$.
\subsection{Main results}
When there is no ambiguity, we omit the symbol $\mathrm{R}$ and write $\pi_*$ (resp.\ $\pi_!$) for the derived pushforward $\mathrm{R}\pi_*$ (resp.\ $\mathrm{R}\pi_!$).
\begin{theorem}\label{1A}
Let $g,n\geq 0$ be integers such that $2g-2+n>0$. Consider the relative good moduli space $\overline{J}_{g,n}^{d,\phi}$ of the universal compactified Jacobians, and the forgetful map $\bar{\pi}\colon \overline{J}_{g,n}^{d,\phi}\to\overline{\mathcal{M}}_{g,n}$. Let $\gamma\colon\mathcal{C}_{g,n}\to \mathcal{M}_{g,n}$ be the universal curve over the smooth locus $\mathcal{M}_{g,n}\subseteq\overline{\mathcal{M}}_{g,n}$. Then, each direct summand of $\bar{\pi}_* \IC({\overline{J}_{g,n}^{d,\phi}})$ must be fully supported on $\overline{\mathcal{M}}_{g,n}$, and there is an isomorphism
    $$\bar{\pi}_* \IC({\overline{J}_{g,n}^{d,\phi}})\simeq \bigoplus_{i=0}^{2g}\IC \big(\overline{\mathcal{M}}_{g,n},(\wedge^i \mathrm{R}^1\gamma_{*}\mathbb{Q}_{\mathcal{C}_{g,n}})\big)[-i+g],$$
in the bounded derived category $D^b\mathrm{MHM}(\overline{\mathcal{M}}_{g,n})$ of mixed Hodge modules on $\overline{\mathcal{M}}_{g,n}$.
\end{theorem}

Although Theorem~\ref{1A} concerns mixed Hodge modules, it suffices to work with perverse sheaves for the proof. Over the smooth locus $\mathcal{M}_{g,n}$, the isomorphism of Hodge modules
$$ \bar{\pi}_*\mathbb{Q}_{\Pic^d(\mathcal{C})} \simeq\bigoplus_{i=0}^{2g}(\wedge^i \mathrm{R}^1{\gamma}_{*}\mathbb{Q}_{\mathcal{C}})[-i]$$
is known. In view of the decomposition theorem for Hodge modules, to prove Theorem~\ref{1A} from the above isomorphism, we only need to verify that every direct summand of $\bar{\pi}_* \IC({\overline{J}_{g,n}^{d,\phi}})$ is fully supported on $\overline{\mathcal{M}}_{g,n}$.

\begin{corollary}\label{1B}

Let $g,n\geq 0$ be integers such that $2g-2+n>0$. Let $J$ be the coarse moduli space of $\overline{J}_{g,n}^{d,\phi}$, and $\pi\colon J\to \overline{M}_{g,n}$ be the forgetful map. Then each direct summand of $\pi_* \IC(J)$ must be fully supported on $\overline{M}_{g,n}$.
     
 \end{corollary}

When the stability condition $\phi$ is nondegenerate, $\overline{\mathcal{J}}_{g,n}^{d,\phi}$ is a smooth Deligne--Mumford stack representable over $\overline{\mathcal{M}}_{g,n}$. In this case, $\overline{J}_{g,n}^{d,\phi}\simeq\overline{\mathcal{J}}_{g,n}^{d,\phi}$, and $\IC(\overline{J}_{g,n}^{d,\phi})$ is simply the shifted constant sheaf $\mathbb{Q}_{\overline{J}_{g,n}^{d,\phi}}[\dim \overline{J}_{g,n}^{d,\phi}]$. Under a nondegenerate stability condition, Theorem~\ref{1A} agrees with the support theorem for relative Jacobians established in \cite[Theorem 5.12]{miglioriniSupportTheoremHilbert2021}. 

As a direct consequence of the theorem \textit{loc.\ cit.},  for all integers $d,d'$, and any nondegenerate stability conditions $\phi,\phi'$, there is a $\mathbb{Q}$-Hodge structure isomorphism:
$$H^*(\overline{J}_{g,n}^{d,\phi},\mathbb{Q})\simeq H^*(\overline{J}_{g,n}^{d',\phi'},\mathbb{Q}).$$

It's easy to see that the right hand side of the isomorphism in Theorem~\ref{1A} is independent of $d,\phi$. By taking global cohomology, we have:
\begin{corollary}\label{1C}
For all $2g-2+n>0$, any integers $d,d'$, and any stability conditions $\phi,\phi'$, there is an isomorphism of intersection cohomology groups:
$$\IH^*(\overline{J}_{g,n}^{d,\phi},\mathbb{Q})\simeq \IH^*(\overline{J}_{g,n}^{d',\phi'},\mathbb{Q}).$$
Moreover, this is an isomorphism of $\mathbb{Q}$-Hodge structures.
\end{corollary}

\subsection{Further remarks}

    When the stability condition $\phi$ is nondegenerate, the cohomology groups are independent of $d$ and $\phi$; however, the cohomology ring structure generally remains dependent on these choices. To address this, the authors in \cite{baeIntrinsicCohomologyRing2025} construct a bigraded algebra, which is a degeneration of the cohomology ring. For $n \geq 1$, this degenerated structure, named the intrinsic cohomology ring, is shown to be independent of both the degree $d$ and the choice of non-degenerate stability condition $\phi$. This motivates the following question:

\begin{question}
    Can the definition of the intrinsic cohomology ring be extended to the setting of degenerate stability conditions?
\end{question}

\subsection{Organization of the paper}
In Section 2, we recall the generalized version of Ng\^{o}'s support theorem by Maulik and Shen \cite{maulikCohomologicalHindependenceModuli2023}, which provides the main technical tool for our proof. In Section 3, we review the construction of the universal compactified Jacobians, and explain why it is necessary to distinguish among moduli stacks with different levels of rigidification. To avoid technical subtleties regarding finite quotient stacks, we base change the forgetful morphism $\bar{\pi}\colon \overline{J}_{g,n}^{d,\phi}\to\overline{\mathcal{M}}_{g,n}$ along a finite \'{e}tale cover $B\to\overline{\mathcal{M}}_{g,n}$. In Section 4, we verify that the base-changed morphism satisfies the hypotheses of the generalized support theorem. Finally, we conclude by descending the resulting decomposition through the stack quotient, utilizing the formalism of the equivariant derived category. In Section 5, we provide an alternative proof of Theorem~\ref{1A}, using \cite[Theorem 5.12]{miglioriniSupportTheoremHilbert2021} and perturbing the stability condition. 

\subsection{Acknowledgments}
The author is grateful to Junliang Shen for proposing this problem and for many helpful discussions and guidance. The author also thanks Andrei Okounkov for his continued support and advice. The author thanks Tianqing Zhu for carefully reading an early draft of this paper and for helpful comments, and Che Shen for helpful discussions.

\section{Ng\^{o} support theorem}
\subsection{Generalized support theorem}\label{s201}

For the main body of this paper, we work over the base field $\mathbb{C}$. Only in Subsection~\ref{s201} do we temporarily set the base field to be a finite field $\mathbf{k}$, with $\bar{\mathbf{k}}$ as its algebraic closure. We assume that $l$ is a prime number coprime to $\mathrm{char}(\mathbf{k})$ when we work with $l$-adic sheaves. The Tate twists are omitted when there is no confusion.

\begin{definition}\label{Defwaf}
    Let $S$ be a scheme over $\mathbf{k}$. Consider a triple $(M,P,S)$ of $\mathbf{k}$-schemes. Let $p\colon P\to S$ be a smooth $S$-group scheme with geometrically connected fibers, and let $\sigma \colon M\to S$ be a proper morphism with a quasi-projective $M$ endowed with a group scheme action
    $$a\colon P\times_{S}M\to M.$$
    
We say that the triple $(M,P,S)$ is a \textit{weak abelian fibration} of relative dimension $d$, if

\begin{enumerate}
    \item every fiber of the map $\sigma$ is pure of dimension $d$, and $M$ has pure dimension 
    $$\dim M=d+\dim S,$$
    \item the action $a$ of $P$ on $M$ has affine stabilizers on each fiber, and
    \item the Tate module $T_{\bar{\mathbb{Q}}_l}(P)$ associated with the group scheme $P$ is polarizable.

\end{enumerate}
\end{definition}

For a closed point $s\in S$, we define $\delta_s$ as the dimension of the affine part of the algebraic group $P_s$. This defines an upper semi-continuous function 
$$\delta\colon S\to \mathbb{N}, s\mapsto\delta_s.$$

For a closed subvariety $Z\subseteq S$, define $\delta_Z$ as the minimal value of the function $\delta$ on $Z$.

We say that $P\to S$ is \textit{$\delta$-regular} if, for any locally closed irreducible subscheme $Z\subseteq S$, we have the inequality 
$$\codim Z\geq \delta_Z.$$

\begin{theorem}\label{MS1}
    Let $(M,P,S)$ be a weak abelian fibration of relative dimension $d$. Let $\mathcal{K}\in D^{b}_c(M,\bar{\mathbb{Q}}_l)$ be a $P$-equivariant bounded complex satisfying the following properties:
    \begin{enumerate}
        \item\label{MS1A}
        \textbf{Decomposition Theorem}
        The direct image complex admits a (non-canonical) decomposition
        $$\sigma_*\mathcal{K}\simeq \bigoplus_i\leftindex^p{\mathcal{H}}^i(\sigma_*\mathcal{K})[-i].$$

        Moreover, after a base change to $S_{\bar{\mathbf{k}}}=S\times_\mathbf{k}\bar{\mathbf{k}}$, the perverse sheaves $\leftindex^p{\mathcal{H}}^i(\sigma_*\mathcal{K})$ are semisimple in the form
        $$\leftindex^p{\mathcal{H}}^i(\sigma_*\mathcal{K})\simeq\bigoplus_{\alpha}\IC({Z_{\alpha,i}},L_{\alpha,i}),$$
        where $Z_{\alpha,i}$ is a closed irreducible subvariety of $S_{\bar{\mathbf{k}}}$ and each $L_{\alpha,i}$ is a pure simple local system of weight $i$ over an open dense subset of  $Z_{\alpha,i}$. We call these $Z_{\alpha,i}$ the supports of the decomposition.
        \item\label{MS1B}
        \textbf{Duality}
        We have an isomorphism 
        $$\mathbb{D}(\mathcal{K})\simeq \mathcal{K}[2\dim M]$$
        with $\mathbb{D}(-)$ as the dualizing functor. 
        \item\label{MS1C}
        \textbf{Relative dimension bound}
        For the standard truncation functor $\tau_{>*}(-)$, we have
        $$\tau_{>2d}(\sigma_*\mathcal{K})=0.$$
    \end{enumerate}
    
    Then for any support $Z$ of the decomposition, we have the inequality 
    $$\codim Z\leq \delta_Z.$$
    
    Moreover, when equality occurs, there exists an open subset $U\subseteq S\otimes_{\mathbf{k}}\bar{\mathbf{k}}$ and a nontrivial local system $L$ on $U\cap Z$ such that $U\cap Z\neq \emptyset$ and $L$ is a direct factor of $\mathrm{R}^{2d}\sigma_* \mathcal{K}\vert_U$. 
\end{theorem}
\begin{proof}
    The first part is exactly \cite[Theorem 1.1]{maulikCohomologicalHindependenceModuli2023}. We now review the proof to show the second half:

    Let $Z$ be a support. Define
    $$\occ(Z)\colonequals\{i\in\mathbb{Z}\vert\leftindex^p{\mathcal{H}}^i(\sigma_*\mathcal{K}) \text{ contains a simple factor with support }Z\},$$
    $$\amp(Z)\colonequals\max(\occ(Z))-\min(\occ(Z)).$$

    As shown by \cite[Proposition 1.3]{maulikCohomologicalHindependenceModuli2023}, for any support $Z$,
    $$\amp(Z)\geq 2(d-\delta_Z).$$
    
    By the duality property \ref{MS1B}, set $\occ(Z)$ is symmetric with respect to the integer $\dim M$. This allows us to pick $m\in\occ(Z)$ with $m\geq \dim M+(d-\delta_Z)$. In particular, we have $\leftindex^p{\mathcal{H}}^m(\sigma_*\mathcal{K})\neq 0$. Thus, by condition (1), there exists an open subset $U\subseteq Z$ and a local system $L$ on $U$ such that
    $$
    (L[\dim Z])[-m]=L[\dim Z-m]
    $$
    is a direct-sum component of the complex $(\sigma_*\mathcal{K})\vert_U$. We obtain 

    $${\mathcal{H}}^{m-\dim Z}(\sigma_*\mathcal{K})\neq 0\in D^b_c(S,\bar{\mathbb{Q}}_l)$$
 for the standard cohomology. Combined with condition (3), we obtain
    $$
    \begin{aligned}
    \codim Z
    &=\dim S-\dim Z && \\
    &=(\dim M-d)-\dim Z && (\text{relative dimension }d)\\
    &=(\dim M-\dim Z)-d && \\
    &\leq \bigl(m-\dim Z-(d-\delta_Z)\bigr)-d && (\text{choice of }m)\\
    &\leq \bigl(2d-(d-\delta_Z)\bigr)-d && (\text{Condition (3)})\\
    &=\delta_Z. &&
    \end{aligned}
    $$

    When equality occurs, $m-\dim Z=2d$. Since $L[\dim Z-m]=L[-2d]$ is a direct-sum component of $(\sigma_*\mathcal{K})\vert_U$, $L$ is a direct factor of $\mathrm{R}^{2d}\sigma_* \mathcal{K}\vert_U$.
\end{proof}

\subsection{Full support theorem}
From now on, we fix the base field to be $\mathbb{C}$. Specifying $\mathcal{K}$ to be the intersection cohomology complex, we obtain:

\begin{theorem}\label{MS2} Let $(M,P,S)$ be a weak abelian fibration of relative dimension $d$. Suppose that 
$$\tau_{>2d}(\sigma_*\IC(M)[-\dim M])=0,$$
and $P\to S$ is $\delta$-regular.

Suppose that $S$ is irreducible, and let $W\subseteq S$ be an open subset such that $\sigma^{-1}(W)\subseteq M$ is smooth and each geometric fiber of $\sigma\vert_{\sigma^{-1}(W)}$ is irreducible. 

Then, if $Z$ is a support of the decomposition of $\sigma_*\IC(M)[-\dim M]$, and $Z\cap W$ is dense in $Z$, we have $Z=S$.
\end{theorem}
\begin{proof}
    The spreading out argument to $\mathbb{C}$ is completely identical to \cite[Theorem 1.8]{maulikCohomologicalHindependenceModuli2023}. The hypothesis $\tau_{>2d}(\sigma_*\IC(M)[-\dim M])=0$ is exactly condition (3) for $\mathcal{K}=\IC(M)[-\dim M]$, so Theorem~\ref{MS1} applies and gives $\codim Z\leq \delta_Z$. Since $P\to S$ is $\delta$-regular, we also have $\codim Z\geq \delta_Z$. Therefore,
    $$\codim Z=\delta_Z.$$

    So for any support $Z$ of decomposition $\sigma_*\IC(M)[-\dim M]$, there exists an open subset $U\subseteq S$ and a nontrivial local system $L$ on $U\cap Z$ such that $U\cap Z\neq \emptyset$, and $L$ is a direct factor of $\mathrm{R}^{2d}\sigma_* \IC(M)[-\dim M]\vert_U$.
    
     Since $\sigma^{-1}(W)$ is smooth, 
    $$\mathrm{R}^{2d}\sigma_* \IC(M)[-\dim M]\vert_{U\cap W}=\mathrm{R}^{2d}\sigma_*\mathbb{Q}_{\sigma^{-1}(U\cap W)}.$$
    
    Therefore, $L\vert_{U\cap W\cap Z}$ is a direct factor of $\mathrm{R}^{2d}\sigma_*\mathbb{Q}_{\sigma^{-1}(U\cap W)}$. Since the geometric fibers of $\sigma\vert_{\sigma^{-1}(W)}$ are irreducible, $\mathrm{R}^{2d}\sigma_*\mathbb{Q}_{\sigma^{-1}(U\cap W)}$ is a simple rank-1 local system; thus
    $$L\vert_{U\cap W\cap Z}=\mathrm{R}^{2d}\sigma_*\mathbb{Q}_{\sigma^{-1}(U\cap W)}.$$

    Since $\mathrm{R}^{2d}\sigma_*\mathbb{Q}_{\sigma^{-1}(U\cap W)}$ is supported on $U\cap W$, we have $U\cap W\cap Z=U\cap W$. Therefore, $Z$ contains the open subset $U\cap W$. Since $S$ is an irreducible scheme, $Z=S$.
\end{proof}

\section{Universal compactified Jacobians}
\subsection{Construction of universal compactified Jacobian stacks}

This subsection is included for background and to fix notation; the detailed stability formalism will not be used in the first proof, but will only be used in the second proof.

Consider stable genus $g$ curves with $n$ marked points, where the natural number pair $(g,n)$ satisfies $2g-2+n>0$. We briefly review the construction of the universal compactified Jacobians by Kass--Pagani \cite{kassStabilitySpaceCompactified2019}.

Let $\mathcal{G}_{g,n}$ denote the set of isomorphism classes of stable $n$-marked graphs of arithmetic genus $g$. For a graph $\Gamma \in \mathcal{G}_{g,n}$, let $V(\Gamma) \colonequals \mathbb{R}^{\VertOp(\Gamma)}$ be the $\mathbb{R}$-vector space of real-valued functions on the vertices of $\Gamma$. For a fixed $d \in \mathbb{R}$, we define the affine hyperplane $V^d(\Gamma) \subseteq V(\Gamma)$ as the set of elements $\phi$ satisfying the degree condition
$$
\sum_{v \in \VertOp(\Gamma)} \phi(v) = d.
$$

We say that an element $\phi \in V^d(\Gamma)$ is \textit{automorphism invariant} if 
$$\phi(v) = \phi(\alpha(v))$$ 
for every $v \in \VertOp(\Gamma)$ and every $\alpha \in \Aut(\Gamma)$.

Suppose that $c\colon \Gamma_1\to \Gamma_2$ is a contraction of stable marked graphs; we say that $\phi(\Gamma_1)\in V(\Gamma_1)$ is \textit{$c$-compatible} with $\phi(\Gamma_2)\in V(\Gamma_2)$ if $$\phi(\Gamma_2)(v_2)=\sum_{c(v_1)=v_2}\phi(\Gamma_1)(v_1)$$ holds for all vertices $v_2\in \VertOp(\Gamma_2)$. 

\begin{definition}
We define $V_{g,n}$ as the affine subspace of vectors in $\prod_{\Gamma\in \mathcal{G}_{g,n}}V^d(\Gamma)$ that are automorphism invariant and compatible with contractions.
\end{definition}

Let $C$ be a stable pointed curve over $\mathbb{C}$ with dual graph $\Gamma$, and $C_0\subseteq C$ be a subcurve with dual graph $\Gamma_0\subseteq \Gamma$. We write $\deg_{\Gamma_0}(F)$ for the total degree of the maximal torsion-free quotient of $F\otimes \mathcal{O}_{C_0}$. We write $\Gamma_0\cap\Gamma _0^c$ for the set of edges that join a vertex of $\Gamma_0$ to a vertex of $\Gamma_0^c$. Given a rank-$1$ torsion-free sheaf $F$ of degree $d$, we have 
$$\deg_{\Gamma_0}(F)+\deg_{\Gamma_0^c}(F)=d-\Delta_{\Gamma_0}(F),$$
where $\Delta_{\Gamma_0}(F)$ is the number of nodes $p\in \Gamma_0\cap \Gamma_0^c$ such that the stalk of $F$ at $p$ is not locally free.

\begin{definition}
    Given $\phi\in V^d(\Gamma)$, we define a rank-$1$ torsion-free sheaf $F$ of degree $d$ on a nodal curve $C$ to be \textit{$\phi$-semistable} (resp. \textit{$\phi$-stable})  if
    $$\Big \vert\deg_{\Gamma_0}(F)-\sum_{v\in \VertOp(\Gamma_0)}\phi(v)+\frac{\Delta_{\Gamma_0}(F)}{2}\Big\vert\leq\frac{\vert\Gamma_0\cap \Gamma_0^c\vert-\Delta_{\Gamma_0}(F) }{2}\quad (\text{resp.}<).$$
    for all $\emptyset\subsetneq\Gamma_0\subsetneq \Gamma$ dual to a proper subcurve $C_0\subsetneq C$.
\end{definition}
We define $\phi\in V^d(\Gamma)$ as \textit{nondegenerate} if every $\phi$-semistable sheaf is $\phi$-stable. We say that $\phi\in V^d_{g,n}$ is nondegenerate if for all $\Gamma\in \mathcal{G}_{g,n}$, the $\Gamma$-component $\phi(\Gamma)$ is nondegenerate in $V^d(\Gamma)$. 
\begin{definition}
 We define $\overline{\mathcal{J}}_{g,n}^{d,\phi,pre}$ as the category fibered in groupoids whose objects are tuples $(C,p_1,\dots,p_n;F)$ consisting of a family of stable $n$-pointed curves $(C/T,p_1,\dots,p_n)$ of genus $g$, and $\phi$-semistable rank-$1$ torsion-free sheaves $F$ of degree $d$ on $C/T$. The morphisms of $\overline{\mathcal{J}}_{g,n}^{d,\phi,pre}$ over a $\mathbb{C}$-morphism $t\colon T\to T'$ are pairs consisting of an isomorphism of curves
    $$\tilde{t}\colon (C,p_1,\dots,p_n)\simeq (C'_T,(p'_1)_T,\dots,(p'_n)_T),$$
    and an isomorphism of $\mathcal O_C$-modules $F\simeq \tilde{t^\ast}F'$.

    Rigidifying $\overline{\mathcal{J}}_{g,n}^{d,\phi,pre}$ by the scalar multiplication group $\mathbb{G}_m$, we obtain a stack $\overline{\mathcal{J}}_{g,n}^{d,\phi}$, and we call it the \textit{$\phi$-compactified universal Jacobian stack}.

\end{definition}

\begin{remark}

While the Kass--Pagani \cite{kassStabilitySpaceCompactified2019} and Melo \cite{meloUniversalCompactifiedJacobians2020} constructions coincide on the stable locus, they differ on the strictly semistable locus because the Kass--Pagani definition includes strictly semistable non-simple sheaves. Because of these non-simple objects, when $\phi$ is degenerate, the resulting stack is strictly Artin rather than Deligne--Mumford. We follow the Kass--Pagani construction in the current work. 

In Subsection~\ref{sec:base-change}, we will construct a Deligne--Mumford stack $\overline{J}_{g,n}^{d,\phi}$ that is the relative good moduli space of $\overline{\mathcal{J}}_{g,n}^{d,\phi}$; such that it rigidifies the fibers while keeping the base still being $\overline{\mathcal{M}}_{g,n}$. 
\end{remark}

\subsection{Good moduli space of universal compactified Jacobian}

Consider the stable locus $\overline{\mathcal{J}}_{g,n}^{d,\phi,sta}\subseteq\overline{\mathcal{J}}_{g,n}^{d,\phi}$, which is a smooth, separated Deligne--Mumford stack \cite[Proposition 4.3]{meloUniversalCompactifiedJacobians2020}. Using Keel--Mori's result \cite[Corollary 1.3]{keelQuotientsGroupoids1997}, the stack admits a coarse moduli space. This argument does not hold for $\overline{\mathcal{J}}_{g,n}^{d,\phi}$, since the stack may not be separated when $\phi$ is degenerate. However, instead of the coarse moduli space, a good moduli space can be constructed using the GIT-theoretic method.

\begin{definition}[\cite{alperGoodModuliSpaces2013}, Definition 4.1]
\label{def:good-moduli}
We say that a quasi-compact morphism $\varphi\colon \mathcal{X}\to \mathcal{Y}$ from an Artin stack $\mathcal{X}$ to an algebraic space $\mathcal{Y}$ is a \textit{good moduli space} if the following properties are satisfied:
\begin{enumerate}
    \item The pushforward functor on quasi-coherent sheaves 
    $$\varphi_*\colon \QCoh(\mathcal{X})\to\QCoh(\mathcal{Y})$$
    is exact.
    \item The natural map $\mathcal{O}_\mathcal{Y}\to \varphi_*\mathcal{O}_{\mathcal{X}}$ is an isomorphism.
\end{enumerate}
\end{definition}

Theorem~\ref{2D} is a specialization of \cite[Theorem A]{cooperGITConstructionsCompactified2024}.

\begin{theorem}\label{2D}
The moduli stack $\overline{\mathcal{J}}_{g,n}^{d,\phi}$ admits a projective good moduli space $J$, which admits a natural morphism to the coarse moduli space of stable curves $\pi\colon J\to\overline{M}_{g,n}$.

Let $\overline{J}(C)$ be the moduli space of rank-$1$, torsion-free, degree $d$, $\phi$-semistable sheaves on $C$. The fiber of $J$ on $[(C,p_1,\dots,p_n)]\in\overline{M}_{g,n}$ is given by $\overline{J}(C)/\Aut(C,p_1,\dots,p_n)$.
\end{theorem}

The projectivity of $J$ follows from Theorem 5.3.1 in \textit{loc.\ cit.} For the compatibility of the definitions, see the remarks following Definition 3.2.1 in \textit{loc.\ cit.}

\begin{proposition}\label{2C}
    The good moduli space $J$ is irreducible of dimension $4g-3+n$.
\end{proposition}
\begin{proof}
    The stack $\overline{\mathcal{J}}_{g,n}^{d,\phi}$ is smooth and irreducible of dimension $4g-3+n$. The argument is completely analogous to \cite[Proposition 3.3]{meloUniversalCompactifiedJacobians2020}: for any reduced curve $C$ with locally planar singularities and any coherent sheaf $F$ on $C$, the deformation functor $\DefOp(C,F)$ is unobstructed.

    Consequently, the good moduli space $J$ is irreducible \cite[Theorem 4.16 (viii)]{alperGoodModuliSpaces2013}. Moreover, since the generic stabilizer of $\overline{\mathcal{J}}_{g,n}^{d,\phi}$ is trivial, the dimension of $J$ is also $4g-3+n$.
\end{proof}

\subsection{Base change to smooth variety}\label{sec:base-change}
    Next, we base change the universal compactified Jacobian over the Teichm\"{u}ller level structured space, which is a smooth variety that is a finite cover of $\overline{M}_{g,n}$. 

\begin{theorem}
    [\cite{arbarelloGeometryAlgebraicCurves2011}, Theorem 16.9.2]
    For each pair of natural numbers $(g,n)$ satisfying $2g-2+n>0$, there exists finite groups $G,H$ such that the moduli space $\leftindex^{}_{G}{\overline{M}}_{g,n}$ of stable $n$-marked genus $g$ curves with Teichm\"{u}ller structure of level $G$ is a smooth variety.
    
    Moreover,
$$\overline{M}_{g,n}=\leftindex^{}_{G}{\overline{M}}_{g,n}/H,$$
and
$$\overline{\mathcal{M}}_{g,n}=[\leftindex^{}_{G}{\overline{M}}_{g,n}/H].$$

\end{theorem}

For the remainder of the paper, we fix an $G$ such that $B\colonequals\leftindex^{}_{G}{\overline{M}}_{g,n}$ is a smooth and connected variety. We have the following natural triangle:

$$
\begin{tikzcd}
& {\overline{\mathcal{M}}_{g,n}} \arrow[dr, "g"] & \\
B \arrow[ur, "f"] \arrow[rr, "h"',pos=0.4] & & {\overline{M}_{g,n}}.
\end{tikzcd}
$$

Define $\widetilde{\mathcal{J}}_{B}^{d,\phi}\colonequals\overline{\mathcal{J}}_{g,n}^{d,\phi}\times_{\overline{\mathcal{M}}_{g,n}}B$ to be the $2$-fiber product. We aim to find the good moduli space of $\widetilde{\mathcal{J}}_{B}^{d,\phi}$. 

The $\phi$-stability condition can be unified with the Gieseker stability condition. For $C\in\overline{\mathcal{M}}_{g,n}$, any $A,M\in\Pic(C)$ with $A$ ample, there exists a stability condition $\phi(A,M)$ such that $F\otimes M$ is Gieseker-semistable with respect to $A$ if and only if $F$ is $\phi$-semistable. Explicitly,

$$\phi(A,M)\colonequals\frac{(d+1-g+m)}{a} \deg(A)+\frac{1}{2} \deg(\omega_C)-\deg(M),$$

where $a,m$ are the total degrees of $A,M$, and $\deg(A),\deg(M)$ refers to the multidegrees. 

On the other hand, given any $\phi\in V^d_{g,n}$, we have the corresponding twisted Gieseker stability condition $(A,M)$.

\begin{proposition}[\cite{kassStabilitySpaceCompactified2019}, Corollary 4.3]
    There exist line bundles $A,M$ on the universal curve $\mathcal{C}\to \overline{\mathcal{M}}_{g,n}$ with $A$ ample relative to $\overline{\mathcal{M}}_{g,n}$ such that, for every stable curve $(C,p_1,\dots,p_n)$, a rank $1$ torsion-free sheaf $F$ of degree $d$ on $C$ is $\phi$-semistable if and only if $F\otimes M$ is $A$-semistable.
\end{proposition}

Set $d'\colonequals d+m$. Since $\overline{\mathcal{J}}^{d}_{g,n}(A,M)\simeq \overline{\mathcal{J}}^{d'}_{g,n}(A,\mathcal{O})$, without loss of generality, we can always assume $M=\mathcal{O}$.

Pulling back via the finite cover $h\colon B\to \overline{\mathcal{M}}_{g,n}$, we can find an ample line bundle $A$ such that $\widetilde{\mathcal{J}}_{B}^{d,\phi,pre}(T)$ consists of $A$-semistable, torsion-free, rank $1$, degree $d$ sheaves over $\tilde{\mathcal{C}}\times_B T\to T$.

\begin{proposition}\label{3B}
    There exists a quasi-projective scheme $Q$ and an integer $N$ such that:
    \begin{enumerate}
        \item $\widetilde{\mathcal{J}}_{B}^{d,\phi}=[Q/\mathrm{PGL}_N]$;
        \item $\widetilde{J}^{d,\phi}_B\colonequals Q//\mathrm{PGL}_N$ is a good moduli space of $\widetilde{\mathcal{J}}_{B}^{d,\phi}$;
        \item $\widetilde{J}^{d,\phi}_B$ is a projective scheme;
        \item the points in $\widetilde{J}^{d,\phi}_B$ represent the S-equivalent classes of $A$-semistable sheaves, $F_1\sim F_2$ if $\gr(F_1)\simeq\gr(F_2)$.
    \end{enumerate}
\end{proposition}
\begin{proof}
   This summarizes the GIT construction in \cite[Theorem 1.21]{simpsonModuliRepresentationsFundamental1994} in the language of stacks. The stack can be written as $\widetilde{\mathcal{J}}_{B}^{d,\phi} = [Q/\mathrm{PGL}_N]$ because it is the $\mathbb{G}_m$-rigidification of the unrigidified stack $\widetilde{\mathcal{J}}_{B}^{d,\phi,pre}(T)=[Q/\mathrm{GL}_N]$. See also \cite[Section 4.3]{huybrechtsGeometryModuliSpaces2010}.
\end{proof}

\begin{proposition}\label{3A}
    The forgetful map $\tilde{\pi}\colon \widetilde{J}^{d,\phi}_B\to B$ is proper, and every fiber is pure of dimension $g$.
\end{proposition}
\begin{proof}
    Since $\widetilde{J}^{d,\phi}_B$ is projective and $B$ is separated, then $\tilde{\pi}$ is a projective map and thus proper.
    
      Over a point $[C]\in B$, the fiber of $\tilde{\pi}$ is the compactified Jacobian $\overline{J}(C)$; by standard deformation theory, it is pure of dimension $g$.
\end{proof}

\begin{lemma}
 There is a natural $H$-action on $\widetilde{J}^{d,\phi}_B$ compatible with the $H$-action on $B$. The map $\bar{h}\colon \widetilde{J}^{d,\phi}_B\to J$ descended from the stack morphism $\bar{\bar{f}}\colon \widetilde{\mathcal{J}}_{B}^{d,\phi}\to \overline{\mathcal{J}}_{g,n}^{d,\phi}$ agrees with the quotient map $\widetilde{J}^{d,\phi}_B\to J=\widetilde{J}^{d,\phi}_B/H$.
    
\end{lemma}
\begin{proof}
There is a natural $H$-action on $\widetilde{J}^{d,\phi}_B$ descended from the $H$-action on $\widetilde{\mathcal{J}}_{B}^{d,\phi}$. Let $\mathrm{h}\in H$, consider the natural right action 
$$R_{\mathrm{h}}\colon B\to B,\quad b\mapsto b\cdot\mathrm{h}$$
whose orbits exactly correspond to points in $\overline{M}_{g,n}$.

Let $[(b,F)]\in \widetilde{J}^{d,\phi}_B$, let $\tilde{R}_{b,\mathrm{h}}\colon \tilde{\mathcal{C}}_{b}\xrightarrow{\sim}\tilde{\mathcal{C}}_{b\cdot \mathrm{h}}$
be the isomorphism of fibers induced by $R_{\mathrm{h}}$. Then $H$ acts on $\widetilde{J}^{d,\phi}_B$ by
$$[(b,F)]\cdot \mathrm{h}=[(b\cdot\mathrm{h},(\tilde{R}_{b,\mathrm{h}})_*F)].$$

Consider any $H$-invariant morphism $\widetilde{J}^{d,\phi}_B\to Z$ to any algebraic space $Z$. It can be uniquely lifted to a $H$-invariant morphism $\widetilde{\mathcal{J}}_{B}^{d,\phi}\to Z$ by composing with the good moduli morphism $\tilde{\rho}\colon \widetilde{\mathcal{J}}_{B}^{d,\phi}\to \widetilde{J}^{d,\phi}_B$. 

The resulting map uniquely descends to a morphism $\overline{\mathcal{J}}_{g,n}^{d,\phi}\to Z$ since $\bar{\bar{f}}\colon \widetilde{\mathcal{J}}^{d,\phi}_B\to \overline{\mathcal{J}}^{d,\phi}_{g,n}$ is an $H$-principal bundle. It further uniquely descends to a morphism $J\to Z$ by the universal property of the good moduli space morphism $\rho\colon\overline{\mathcal{J}}_{g,n}^{d,\phi}\to J$. 

To conclude, any $H$-invariant map $\widetilde{J}^{d,\phi}_B\to Z$ uniquely factors through $\bar{h}\colon \widetilde{J}^{d,\phi}_B\to J$, so that $\bar{h}$ is the quotient map.
\end{proof}

\begin{definition}
    We define the \textit{universal compactified Jacobian space} $\overline{J}_{g,n}^{d,\phi}$ to be the quotient stack $[\widetilde{J}^{d,\phi}_B/H]$. It is a Deligne--Mumford stack, and $J=\widetilde{J}^{d,\phi}_B/H$ is its coarse moduli space.
\end{definition}

Since 
$$\bar{\bar{f}}\colon \widetilde{\mathcal{J}}_{B}^{d,\phi}\to \overline{\mathcal{J}}_{g,n}^{d,\phi}\simeq[\widetilde{\mathcal{J}}_{B}^{d,\phi}/H]$$ 
is an $H$-principal bundle, then the map $\tilde{\rho}\colon \widetilde{\mathcal{J}}_{B}^{d,\phi}\to\widetilde{J}^{d,\phi}_B$ naturally descends to a map
$$\bar{\rho}\colon \overline{\mathcal{J}}_{g,n}^{d,\phi}\simeq[\widetilde{\mathcal{J}}_{B}^{d,\phi}/H]\to [\widetilde{J}^{d,\phi}_B/H]=\overline{J}_{g,n}^{d,\phi}.$$

Similarly, $\tilde{\pi}\colon \widetilde{J}^{d,\phi}_B\to B$ descends to $\bar{\pi}\colon \overline{J}_{g,n}^{d,\phi}\to \overline{\mathcal{M}}_{g,n}.$
$$
\begin{tikzcd}[row sep=3em, column sep=3.5em]
& \overline{\mathcal{J}}_{g,n}^{d,\phi} \arrow[d, "\bar{\rho}" left] \arrow[ddr, "\rho"] & \\
\widetilde{\mathcal{J}}_{B}^{d,\phi} \arrow[ur, "\bar{\bar{f}}"] \arrow[d, "\tilde{\rho}" left] & \overline{J}_{g,n}^{d,\phi} \arrow[dd, "\bar{\pi}" left, pos=0.25] \arrow[dr, "\bar{g}"] & \\
\widetilde{J}^{d,\phi}_B \arrow[rr, crossing over, "\bar{h}",pos=0.6] \arrow[dd, "\tilde{\pi}" left] \arrow[ur, "\bar{f}"] & & J \arrow[dd, "\pi"] \\
& {\overline{\mathcal{M}}_{g,n}} \arrow[dr, "g"] & \\
B \arrow[ur, "f"] \arrow[rr, "h",pos=0.6] & & {\overline{M}_{g,n}}
\end{tikzcd}
$$

We summarize the maps that we have defined in the diagram above. 

\begin{definition}
    Let $\varphi\colon \mathcal{X}\to \mathcal{Y}$ be a quasi-compact morphism between Artin stacks with a common base stack $\mathcal{S}$. We say $\mathcal{Y}$ is a \textit{relative good moduli space} of $\mathcal{X}$ relative to $\mathcal{S}$ if $\mathcal{Y}\to\mathcal{S}$ is representable by algebraic spaces, and conditions (1) and (2) of Definition~\ref{def:good-moduli} for $\varphi\colon \mathcal{X}\to \mathcal{Y}$ are satisfied.
\end{definition}
\begin{proposition}
    The map $\bar{\rho}\colon\overline{\mathcal{J}}_{g,n}^{d,\phi}\to\overline{J}_{g,n}^{d,\phi}$ is the relative good moduli space relative to $\overline{\mathcal{M}}_{g,n}$.
\end{proposition}
\begin{proof}
 Since $\bar{f}\colon\widetilde{J}^{d,\phi}_B\to\overline{J}_{g,n}^{d,\phi}$ is an \'{e}tale cover, it is fpqc. Since $\tilde{\rho}$ is a good moduli space morphism, the base changed map $\bar{\rho}$ satisfies conditions (1) and (2) of Definition~\ref{def:good-moduli} from the proof of \cite[Proposition 3.10(v)]{alperGoodModuliSpaces2013}. 

On the other hand, since the pulled-back morphism of $\bar{\pi}$ under the $2$-cartesian square is a scheme morphism $\tilde{\pi}\colon \widetilde{J}^{d,\phi}_B \to B$, then $\bar{\pi}$ is representable.
\end{proof}

\section{Proof of the main theorem}

From now on, we set $P\colonequals\Pic^0(\tilde{\mathcal{C}}/B)$ as the multidegree $0$ Picard variety. It is a smooth group scheme on $B$; each fiber is geometrically connected with pure dimension $g$. There is a natural $B$-group scheme action $P$ on $\widetilde{J}^{d,\phi}_B$. 

We would like to check that $(\widetilde{J}^{d,\phi}_B,P,B)$ satisfies the conditions of Theorem~\ref{MS2} in the following sections.

\subsection{Weak abelian fibration}

\begin{proposition}\label{prop:waf}
   The triple $(\widetilde{J}^{d,\phi}_B,P,B)$ is a weak abelian fibration of relative dimension $g$.
\end{proposition}
\begin{proof}
    By Proposition~\ref{3A}, $\tilde{\pi}$ is proper, and the fibers are pure of dimension $g$. Because $J$ is irreducible of dimension $4g-3+n$ (Proposition~\ref{2C}) and $J=\widetilde{J}^{d,\phi}_B/H$, the irreducible components of $\widetilde{J}^{d,\phi}_B$ are isomorphic in dimension $4g-3+n$. On the other hand, $B$, as a finite cover of $\overline{M}_{g,n}$, has dimension $3g-3+n$. This verifies the first condition of Definition~\ref{Defwaf}.

    For a stable curve $C=\sum C_i$ with irreducible components $C_i$ and a point $[(C,F)]\in \widetilde{J}^{d,\phi}_B$, the stabilizer lies in the kernel of the natural morphism
    $$\Pic^0(C)\to \prod \Pic^0(C_i).$$

    The kernel of this map is exactly the affine part of $\Pic^{0}(C)$. This verifies the second condition.

    Since $P$ is a quasi-projective group scheme over $B$, the third condition follows from \cite[Theorem 1.2]{anconaNgoSupportTheorem2024}, which states that every Tate module associated with a quasi-projective group scheme is polarizable.
\end{proof}

In the current case, $\delta_C=\dim\Pic^{0}(C)^{\mathrm{aff}}=b_1(\Gamma)$, that is the first Betti number of the dual graph $\Gamma$.

\subsection{Relative dimension bound}
The main purpose of this subsection is to prove the relative dimension bound condition: $$\tau_{>2g}(\tilde{\pi}_*\IC({\widetilde{J}^{d,\phi}_B})[-\dim \widetilde{J}^{d,\phi}_B])=0.$$

\begin{theorem}[\cite{kinjoDecompositionTheoremGood2024}, Theorem 5.1]\label{Kin}
    Let $\mathcal{X}$ be a smooth Artin stack with affine diagonal admitting a good moduli space $\varphi\colon \mathcal{X}\to \mathcal{Y}$. Then the mixed Hodge complex $\varphi_*\mathbb{Q}_{\mathcal{X}}$ is pure. In particular, there exists an isomorphism
    $$\varphi_*\mathbb{Q}_{\mathcal{X}}\simeq \bigoplus_{i}\leftindex^p{\mathcal{H}}^i(\varphi_*\mathbb{Q}_{\mathcal{X}})[-i].$$
\end{theorem}

\begin{lemma}\label{lem:affine-diagonal}
$\widetilde{\mathcal{J}}_{B}^{d,\phi}$ has an affine diagonal.
\end{lemma}
\begin{proof}
If $G$ is an affine group scheme acting on an algebraic space $X$ that has an affine diagonal, then the quotient stack $[X/G]$ has an affine diagonal \cite[Example 12.10]{alperGoodModuliSpaces2013}. Recall from Theorem~\ref{3B} that $\widetilde{\mathcal{J}}_{B}^{d,\phi}\simeq[Q/\mathrm{PGL}_N]$ with $Q$ is a quasi-projective scheme. Since $Q$ is quasi-projective, it is separated; consequently, its diagonal is a closed immersion and thus affine. We conclude that $\widetilde{\mathcal{J}}_{B}^{d,\phi}$ has an affine diagonal.
\end{proof}

\begin{proposition}\label{5B}
 The complex $\IC({\widetilde{J}^{d,\phi}_B})[-\dim \widetilde{J}^{d,\phi}_B]$ is a direct summand of $\tilde{\rho}_*\mathbb{Q}_{\widetilde{\mathcal{J}}_{B}^{d,\phi}}$.
\end{proposition}

\begin{proof}
    Since $\widetilde{\mathcal{J}}_{B}^{d,\phi}$ is smooth, has an affine diagonal, and admits the good moduli space morphism $\tilde{\rho}\colon \widetilde{\mathcal{J}}_{B}^{d,\phi}\to \widetilde{J}^{d,\phi}_B$ , we may apply Theorem~\ref{Kin}. Hence, the mixed Hodge complex $\tilde{\rho}_*\mathbb{Q}_{\widetilde{\mathcal{J}}_{B}^{d,\phi}}$ is pure, and we have the following decomposition:
    $$\tilde{\rho}_*\mathbb{Q}_{\widetilde{\mathcal{J}}_{B}^{d,\phi}}\simeq \bigoplus_{i}\leftindex^p{\mathcal{H}}^i(\tilde{\rho}_*\mathbb{Q}_{\widetilde{\mathcal{J}}_{B}^{d,\phi}})[-i].$$
    
Since the mixed Hodge complex $\tilde{\rho}_*\mathbb{Q}_{\widetilde{\mathcal{J}}_{B}^{d,\phi}}$ is pure, it is semisimple, and we can further decompose:
$$\tilde{\rho}_*\mathbb{Q}_{\widetilde{\mathcal{J}}_{B}^{d,\phi}}\simeq \bigoplus_{i}\bigoplus_{\alpha}\IC(Z_\alpha,\mathcal{L}_\alpha)[-i],$$
where $Z_\alpha$ are irreducible closed subvarieties, and $\mathcal{L}_\alpha$ are local systems over an open subset of $Z_\alpha$. 

Let $U\colonequals (h\circ\tilde{\pi})^{-1}(M_{g,n})\subseteq\widetilde{J}^{d,\phi}_B$ be the open locus over smooth curves. Since every semistable sheaf over a smooth curve $C$ is simple, any point in $\tilde{\rho}^{-1}(U)\subseteq\widetilde{\mathcal{J}}_{B}^{d,\phi}$ has a trivial stabilizer. Therefore,
$$\tilde{\rho}\vert_{\tilde{\rho}^{-1}(U)}\colon \tilde{\rho}^{-1}(U)\to U$$
is an isomorphism. Moreover,
$$\tilde{\rho}_*\mathbb{Q}_{\tilde{\rho}^{-1}(U)}\simeq \mathbb{Q}_{U}.$$

Since $U\subseteq \widetilde{J}^{d,\phi}_B$ is smooth and dense, and $\IC(\widetilde{J}^{d,\phi}_B)[-\dim \widetilde{J}^{d,\phi}_B]$ is the unique minimal extension of the constant local system $\mathbb{Q}_{U}$, we conclude that $\IC(\widetilde{J}^{d,\phi}_B)[-\dim \widetilde{J}^{d,\phi}_B]$ must appear as a direct summand of $\tilde{\rho}_*\mathbb{Q}_{\widetilde{\mathcal{J}}_{B}^{d,\phi}}$.
\end{proof}

\begin{proposition}\label{5E}
Define $\tilde{u}\colonequals\tilde{\pi}\circ \tilde{\rho} \colon \widetilde{\mathcal{J}}_{B}^{d,\phi}\to B$, then 
   $$\tau_{>2g}(\tilde{u}_!\mathbb{Q}_{\widetilde{\mathcal{J}}_{B}^{d,\phi}})=0.$$
\end{proposition}
\begin{proof}
    Let $b\in B$ be a closed point. We denote by $(\widetilde{\mathcal{J}}_{B}^{d,\phi})_b$ the substack
   $$(\widetilde{\mathcal{J}}_{B}^{d,\phi})_b\colonequals\tilde{u}^{-1}(b)\subseteq \widetilde{\mathcal{J}}_{B}^{d,\phi}.$$
   
   We have $\dim(\widetilde{\mathcal{J}}_{B}^{d,\phi})_b\leq \dim(\widetilde{J}_{B}^{d,\phi})_b=g$. By \cite[Lemma 3.5]{maulikCohomologicalHindependenceModuli2023}, which states that the compact support cohomology for an Artin stack of dimension $g$ is concentrated in degrees $\leq 2g$, we conclude that the complex
$$(\tilde{u}_!\mathbb{Q}_{\widetilde{\mathcal{J}}_{B}^{d,\phi}})_b=H^*_c((\widetilde{\mathcal{J}}_{B}^{d,\phi})_b,\mathbb{Q})$$
   is concentrated in degrees $\leq 2g$ for any closed point $b\in B$.
\end{proof}
We prove the relative dimension bound:
\begin{proposition}\label{5D}
    Consider the weak abelian fibration triple $(\widetilde{J}^{d,\phi}_B,P,B)$ and the map $\tilde{\pi}\colon \widetilde{J}^{d,\phi}_B\to B$. We have
    $$\tau_{>2g}(\tilde{\pi}_*\IC({\widetilde{J}^{d,\phi}_B})[-\dim \widetilde{J}^{d,\phi}_B])=0.$$
\end{proposition}
\begin{proof}
 From Proposition~\ref{5B}, we know that there exists a splitting
    
   $$\tilde{\rho}_*\mathbb{Q}_{\widetilde{\mathcal{J}}_{B}^{d,\phi}}\simeq \IC({\widetilde{J}^{d,\phi}_B})[-\dim \widetilde{J}^{d,\phi}_B]\oplus \mathcal{E},\quad \mathcal{E}\in D^b(\widetilde{J}^{d,\phi}_B,\mathbb{Q}).$$
   
 Applying the dualizing functor to the isomorphism, we obtain the following:
   
   $$\mathbb{D}(\tilde{\rho}_*\mathbb{Q}_{\widetilde{\mathcal{J}}_{B}^{d,\phi}})\simeq \IC({\widetilde{J}^{d,\phi}_B})[\dim \widetilde{J}^{d,\phi}_B]\oplus \mathbb{D}(\mathcal{E}).$$
   
   Since $\widetilde{\mathcal{J}}_{B}^{d,\phi}$ is nonsingular, the lefthand side is isomorphic to
   
   $$
   \begin{aligned}
   \tilde{\rho}_!\mathbb{D}(\mathbb{Q}_{\widetilde{\mathcal{J}}_{B}^{d,\phi}})
   &= \tilde{\rho}_!\mathbb{Q}_{\widetilde{\mathcal{J}}_{B}^{d,\phi}}[2\dim \widetilde{\mathcal{J}}_{B}^{d,\phi}] \\
   &= \tilde{\rho}_!\mathbb{Q}_{\widetilde{\mathcal{J}}_{B}^{d,\phi}}[2\dim \widetilde{J}^{d,\phi}_B] .
   \end{aligned}
   $$
   
   Combining the above equations, we conclude that
   
   $$\tilde{\rho}_!\mathbb{Q}_{\widetilde{\mathcal{J}}_{B}^{d,\phi}}\simeq \IC({\widetilde{J}^{d,\phi}_B})[-\dim \widetilde{J}^{d,\phi}_B]\oplus \mathcal{E}',\quad \mathcal{E}'\in D^b(\widetilde{J}^{d,\phi}_B,\mathbb{Q}).$$
   
   Since the map $\tilde{\pi}\colon \widetilde{J}^{d,\phi}_B\to B$ is proper, we have $\tilde{\pi}_!=\tilde{\pi}_*$. Hence $\tilde{\pi}_*\IC({\widetilde{J}^{d,\phi}_B})[-\dim \widetilde{J}^{d,\phi}_B]$ is a direct summand of $\tilde{\pi}_*(\tilde{\rho}_!\mathbb{Q}_{\widetilde{\mathcal{J}}_{B}^{d,\phi}})$, and moreover
   $$
   \begin{aligned}
   \tilde{\pi}_*(\tilde{\rho}_!\mathbb{Q}_{\widetilde{\mathcal{J}}_{B}^{d,\phi}})
   &= \tilde{\pi}_!(\tilde{\rho}_!\mathbb{Q}_{\widetilde{\mathcal{J}}_{B}^{d,\phi}}) \\
   &= \tilde{u}_!\mathbb{Q}_{\widetilde{\mathcal{J}}_{B}^{d,\phi}}.
   \end{aligned}
   $$
   
   Finally, combining with Proposition~\ref{5E}, we can conclude that
   
   $$\tau_{>2g}(\tilde{\pi}_*\IC({\widetilde{J}^{d,\phi}_B})[-\dim \widetilde{J}^{d,\phi}_B])=0.$$
\end{proof}

\subsection{\texorpdfstring{$\delta$}{delta}-regularity}

Consider the stratification of $B$ by the dual graphs 
$$B=\bigsqcup_{\Gamma\in\mathcal{G}_{g,n}} B_{\Gamma},$$
where each $B_{\Gamma}$ is a locally closed subset of $B$, parameterizing curves whose dual graph is $\Gamma$. Passing from the finite cover $B\to \overline{\mathcal{M}}_{g,n}$, we have $\codim B_{\Gamma}=\vert\EdgeOp(\Gamma)\vert$ from the underlying stratification (see \cite[Section 12.10]{arbarelloGeometryAlgebraicCurves2011}).

\begin{proposition}\label{6A}
    Let $Z \subseteq B$ be an irreducible locally closed subvariety with generic point $\eta \in Z$, and let $\Gamma$ be the dual graph of the corresponding generic curve $C_{\eta}$. Then the codimension of $Z$ satisfies
        $$\codim Z \geq \delta_Z + \vert\VertOp(\Gamma)\vert - 1.$$
        
        Consequently, $\codim Z \geq \delta_Z$ for all such $Z$, meaning the group scheme $P\to B$ is $\delta$-regular.
\end{proposition}

\begin{proof}
    By definition, the generic point $\eta$ lies in the stratum $B_{\Gamma}$. Consequently, $Z$ is contained in the closure $\overline{B_{\Gamma}}$, which implies
    $$\codim Z \geq \codim B_{\Gamma} = \vert\EdgeOp(\Gamma)\vert.$$

    Recall that the affine rank of the Picard group of the generic curve is given by the first Betti number of its dual graph:
    $$\delta_{C_{\eta}} = \dim \Pic^{0}(C_{\eta})^{\mathrm{aff}} = b_1(\Gamma) = \vert\EdgeOp(\Gamma)\vert - \vert\VertOp(\Gamma)\vert + 1.$$
    
    Because the function $\delta$ is upper semicontinuous on the irreducible space $Z$, its minimum value $\delta_Z$ is achieved at the generic point, yielding $\delta_Z = \delta_{C_{\eta}}$.
    
      Substituting this into our codimension inequality yields
    $$\codim Z \geq \vert\EdgeOp(\Gamma)\vert = \delta_Z + \vert\VertOp(\Gamma)\vert - 1.$$
\end{proof}

\subsection{Full support theorem over \texorpdfstring{$B$}{B}}\label{sec:full-support-B}
\begin{theorem}\label{FullSupportOverB}
    Let $B^{\circ}\colonequals f^{-1}(\mathcal{M}_{g,n})\subseteq B$ be the smooth locus in $B$. Let $\tilde{\gamma}^{\circ}\colon \tilde{\mathcal{C}}^{\circ}\to B^{\circ}$ be the pull-backed universal curve over $B^\circ$. Consider the map $\tilde{\pi}\colon \widetilde{J}^{d,\phi}_B\to B$. We have the following:
    $$\tilde{\pi}_* \IC({\widetilde{J}^{d,\phi}_B})=\bigoplus_{i=0}^{2g}\IC\big(B,(\wedge^i \mathrm{R}^1\tilde{\gamma}^{\circ}_{*}\mathbb{Q}_{\tilde{\mathcal{C}}^{\circ}})\big)[-i+g].$$
\end{theorem}
\begin{proof}
    Recall a well-known fact that 
    $$\tilde{\pi}_*\mathbb{Q}_{\Pic^d(\tilde{\mathcal{C}}^{\circ}/B^{\circ})}\simeq \bigoplus_{i=0}^{2g}(\wedge^i \mathrm{R}^1\tilde{\gamma}^{\circ}_{*}\mathbb{Q}_{\tilde{\mathcal{C}}^{\circ}})[-i]$$
 as variation of Hodge structures (see, \cite[Lemma 1.3.5]{decataldoTopologyHitchinSystems2012}). Since $\tilde{\pi}^{-1}(B^{\circ})\simeq \Pic^d(\tilde{\mathcal{C}}^{\circ}/B^{\circ})$ is a smooth subscheme of $\widetilde{J}^{d,\phi}_B$,
    $$\tilde{\pi}_* \IC({\widetilde{J}^{d,\phi}_B})\big\vert_{B^{\circ}}\simeq \tilde{\pi}_*\mathbb{Q}_{\Pic^d(\tilde{\mathcal{C}}^{\circ}/B^{\circ})}[\dim \widetilde{J}^{d,\phi}_B]\simeq \bigoplus_{i=0}^{2g}(\wedge^i \mathrm{R}^1\tilde{\gamma}^{\circ}_{*}\mathbb{Q}_{\tilde{\mathcal{C}}^{\circ}})[-i+\dim\widetilde{J}^{d,\phi}_B]. $$

 Assume that there is an irreducible closed subvariety $Z\subseteq B$ that is a support for $\tilde{\pi}_* \IC({\widetilde{J}^{d,\phi}_B})$. To prove the theorem, it is equivalent to showing that $Z$ must equal $B$. Let $\Gamma$ be the dual graph of the generic curve of $Z$. By Theorem~\ref{MS1} and Proposition~\ref{6A},
    $$\codim Z\leq \delta_Z\leq \codim Z+1-\vert\VertOp(\Gamma)\vert.$$

 Therefore, it forces $\delta_Z=\codim Z$ and $\vert\VertOp(\Gamma)\vert=1$, which is equivalent to saying that $Z$ is generically integral. 
 
Let $W=f^{-1}(\overline{\mathcal{M}}^{int}_{g,n})\subseteq B$ be the locus of integral curves. Since over an integral curve, all rank-one torsion-free sheaves are simple, so $\tilde{\pi}^{-1}(W)$ is isomorphic to the corresponding substack $\tilde{\rho}^{-1}(\tilde{\pi}^{-1}(W))$. Since the ambient stack $\widetilde{\mathcal{J}}^{d,\phi}_{B}$ is smooth, we conclude that $\tilde{\pi}^{-1}(W)$ is a smooth variety. 

Since $Z$ is generically integral, $W\cap Z$ is dense in $Z$. Recall that, for each integral nodal curve $C$, the corresponding fiber $\bar{J}(C)$ is irreducible \cite[Theorem~A]{Rego1980}. Apply Theorem~\ref{MS2} and we reach the desired conclusion.
\end{proof}

\subsection{Descent to \texorpdfstring{$\overline{\mathcal{M}}_{g,n}$}{Mg,n}}\label{sec:proof1}
In order to descend the full support result to $\bar\pi\colon \overline{J}_{g,n}^{d,\phi}\to \overline{\mathcal{M}}_{g,n}$, we apply the tools of the equivariant derived category.

\begin{corollary}\label{6B}

Consider $\wedge^i \mathrm{R}^1\tilde{\gamma}^{\circ}_{*}\mathbb{Q}_{\tilde{\mathcal{C}}^{\circ}}$ as an $H$-equivariant local system induced by the $H$-equivariant universal curve $\tilde{\gamma}^{\circ}\colon\tilde{\mathcal{C}}^{\circ}\to B^{\circ}$. Then,

    $$\tilde{\pi}_* \IC_H({\widetilde{J}^{d,\phi}_B})\simeq\bigoplus_{i=0}^{2g}\IC_H\big(B,(\wedge^i \mathrm{R}^1\tilde{\gamma}^{\circ}_{*}\mathbb{Q}_{\tilde{\mathcal{C}}^{\circ}})\big)[-i+g].$$
    
\end{corollary}
\begin{proof}

By the $H$-equivariant decomposition theorem, $\tilde{\pi}_* \IC_H(\widetilde{J}^{d,\phi}_B)$ admits a decomposition into shifted equivariant intersection complexes
$$
\tilde{\pi}_* \IC_H(\widetilde{J}^{d,\phi}_B)\simeq \bigoplus_{\alpha}\IC_H(Z_{\alpha},L_{\alpha})[m_{\alpha}],
$$
where each $Z_{\alpha}\subset B$ is a closed irreducible $H$-invariant subvariety and $L_{\alpha}$ is an $H$-equivariant semisimple local system on a dense open subset of $Z_{\alpha}$.

Applying the forgetful functor $\For:D^b_H(B)\to D^b_c(B)$, we obtain
$$
\For(\tilde{\pi}_* \IC_H(\widetilde{J}^{d,\phi}_B))
\simeq
\tilde{\pi}_*\IC(\widetilde{J}^{d,\phi}_B)
\simeq
\bigoplus_{\alpha}\IC(Z_{\alpha},\For(L_{\alpha}))[m_{\alpha}].
$$

By Theorem~\ref{FullSupportOverB}, every support occurring in
$\tilde{\pi}_*\IC(\widetilde{J}^{d,\phi}_B)$ is equal to $B$. Hence, every
$Z_{\alpha}$ above is equal to $B$. Therefore, all summands of $\tilde{\pi}_* \IC_H(\widetilde{J}^{d,\phi}_B)$ have full
support on $B$.

Moreover, since the identification

$$\tilde{\pi}_* \IC_H({\widetilde{J}^{d,\phi}_B})\big\vert_{B^{\circ}}\simeq \tilde{\pi}_*\mathbb{Q}_{\Pic^d(\tilde{\mathcal{C}}^{\circ}/B^{\circ})}[\dim \widetilde{J}^{d,\phi}_B]\simeq \bigoplus_{i=0}^{2g}(\wedge^i \mathrm{R}^1\tilde{\gamma}^{\circ}_{*}\mathbb{Q}_{\tilde{\mathcal{C}}^{\circ}})[-i+\dim\widetilde{J}^{d,\phi}_B]$$
is $H$-equivariant, we obtain this corollary.
\end{proof}

We prove the main theorem:
\begin{proof}[Proof of Theorem~\ref{1A}]
For any quotient stack $[Y/G]$, its constructible derived category is defined by the $G$-equivariant constructible derived category of $Y$~\cite[Section 6.8]{acharPerverseSheavesApplications2021}. We denote the identity functor by:
$$Q\colon D^b_G(Y)\xrightarrow{\simeq} D^b_c([Y/G]).$$

Apply $Q$ to both sides of the equation in Corollary~\ref{6B}. By functoriality of $Q$, the left hand side:
$$Q(\tilde{\pi}_*\IC_H(\widetilde{J}^{d,\phi}_B))=\bar{\pi}_*Q\IC_H(\widetilde{J}^{d,\phi}_B)=\bar{\pi}_*\IC(\overline{J}_{g,n}^{d,\phi}).$$

On the other hand, the universal curve $\tilde{\gamma}^{\circ}\colon \tilde{\mathcal{C}}^{\circ}\to B^{\circ}=f^{-1}(\mathcal{M}_{g,n})$ is the pullback of the universal curve $\gamma\colon \mathcal{C}_{g,n}\to \mathcal{M}_{g,n}$. Therefore,
$$\wedge^i \mathrm{R}^1\tilde{\gamma}^{\circ}_{*}\mathbb{Q}_{\tilde{\mathcal{C}}^{\circ}}=f^*(\wedge^i\mathrm{R}^1\gamma_{*}\mathbb{Q}_{\mathcal{C}_{g,n}}),$$
and we can conclude that the right hand side:
$$Q\Big(\bigoplus_{i=0}^{2g}\IC_H\big(B,(\wedge^i \mathrm{R}^1\tilde{\gamma}^{\circ}_{*}\mathbb{Q}_{\tilde{\mathcal{C}}^{\circ}})\big)[-i+g]\Big)=
\bigoplus_{i=0}^{2g}\IC\big(\overline{\mathcal{M}}_{g,n},(\wedge^i \mathrm{R}^1\gamma_{*}\mathbb{Q}_{\mathcal{C}_{g,n}})\big)[-i+g].$$
\end{proof}

\begin{lemma}\label{lem:IC-quo-stack}
    Let $G$ be a finite group acting on $Y$, a separated and finite type scheme over $\mathbb{C}$. Consider the natural map $\psi\colon [Y/G]\to Y/G$ from the quotient stack to its coarse moduli space. Let $V\subseteq [Y/G]$ be an open dense subspace. Let $\mathcal{L}$ be a local system supported on $V$. Then there is an isomorphism:
    $$\psi_*\IC([Y/G],\mathcal{L})\simeq \IC(Y/G,\psi_*\mathcal{L}).$$
\end{lemma}
\begin{proof}
    This lemma is a generalized version of \cite[Theorem~8.7.1]{bernsteinEquivariantSheavesFunctors1994}.
    
    Let $\bar{\psi}\colon Y\to Y/G$ be the finite quotient map. Let $\tilde{\mathcal{L}}$ be the $G$-equivariant local system on the preimage $\tilde{V} = \bar{\psi}^{-1}(\psi(V)) \subseteq Y$ that corresponds to the local system $\mathcal{L}$ on $V \subseteq [Y/G]$.

    Translating our goal into the equivariant derived category, we need to show that: 
    $$\big(\bar{\psi}_*\IC_G(Y,\tilde{\mathcal{L}})\big)^G \simeq \IC\big(Y/G, (\bar{\psi}_*\tilde{\mathcal{L}})^G \big).$$
       
    Choose a stratification $\mathcal{S}$ of $Y/G$ and let $\mathcal{T}=\bar{\psi}^{-1}(\mathcal{S})$ be the induced stratification of $Y$, such that both intersection cohomology complexes can be constructed by iteratively taking pushforwards and truncations along the strata, starting from the open subspaces $\psi(V)$ and $\tilde{V}$.
    
    Since $G$ is finite, taking $G$-invariants is an exact functor, and it commutes with both the pushforward and truncation functors. Applying this inductively to the construction of the intersection cohomology complexes yields the desired isomorphism.
\end{proof}

\begin{proof}[Proof of Corollary~\ref{1B}]
Let $g\colon \overline{\mathcal{M}}_{g,n}\to \overline{M}_{g,n}$ and $\bar{g}\colon \overline{J}_{g,n}^{d,\phi}\to J$ be the coarse moduli space morphisms. By functoriality, we have $\pi\circ \bar{g}=g\circ \bar{\pi}$, hence
$$\pi_*\IC(J)\simeq \pi_*\bar{g}_*\IC(\overline{J}_{g,n}^{d,\phi})\simeq g_*\bar{\pi}_*\IC(\overline{J}_{g,n}^{d,\phi}).$$

By Theorem~\ref{1A} and Lemma~\ref{lem:IC-quo-stack}, 
$$
\begin{aligned}
g_*\bar{\pi}_*\IC(\overline{J}_{g,n}^{d,\phi})
\simeq{}& g_*\Bigl(\bigoplus_{i=0}^{2g}\IC\bigl(\overline{\mathcal{M}}_{g,n},(\wedge^i \mathrm{R}^1\gamma_{*}\mathbb{Q}_{\mathcal{C}_{g,n}})\bigr)[-i+g]\Bigr)\\
\simeq{}& \bigoplus_{i=0}^{2g}\IC\bigl(\overline{M}_{g,n},g_*(\wedge^i \mathrm{R}^1\gamma_{*}\mathbb{Q}_{\mathcal{C}_{g,n}})\bigr)[-i+g].
\end{aligned}
$$
Therefore each direct summand of $\pi_*\IC(J)$ is of full support.
\end{proof}

\section{The Second Proof}

\subsection{The \texorpdfstring{$n\geq 1$}{n>=1} case}\label{sec:ngeq1}
We assume that $n \geq 1$ in this subsection. 

When the stability condition $\phi$ is nondegenerate, our main theorem can be deduced by \cite[Theorem 5.12]{miglioriniSupportTheoremHilbert2021}. Recall from \cite[Section 5]{kassStabilitySpaceCompactified2019} that the universal stability space $V^d_{g,n}$ has a wall-and-chamber structure. When $\phi$ is degenerate, it sits on the walls, which are the union of locally finitely many hyperplanes of $V^{d}_{g,n}$. When $\phi$ is nondegenerate, it resides in chambers that are complementary to the walls. We call each connected chamber a \textit{stability prototype}; we denote $\mathcal{P}(\phi)$ as the unique stability prototype that $\phi$ resides in.

From Remark 5.12 of \textit{loc.\ cit.}, when $n\geq 1$, nondegenerate stability conditions always exist. We can therefore choose a nondegenerate $\phi'$ close enough to $\phi$, such that $\phi\in \overline{\mathcal{P}(\phi')}$. By \cite[Lemma 3.20]{kassExtensionsUniversalTheta2017}, every $\phi'$-stable sheaf is $\phi$-semistable. This inclusion induces a natural open immersion of the corresponding moduli stacks
    $$ \iota \colon \overline{\mathcal{J}}_{g,n}^{d,\phi'} \hookrightarrow \overline{\mathcal{J}}_{g,n}^{d,\phi}.$$
    
    Passing to the relative good moduli spaces, we have a map $r\colon \overline{J}^{d,\phi'}_{g,n} \to \overline{J}^{d,\phi}_{g,n}$. 

\begin{proposition}\label{prop:rsurjective}
    The map $r$ is proper and surjective, compatible with the forgetful map to $\overline{\mathcal{M}}_{g,n}$.
\end{proposition}
\begin{proof}
    Since $\iota$ commutes with forgetful maps to $\overline{\mathcal{M}}_{g,n}$, $r$ also commutes with the forgetful maps. For the forgetful morphism $\bar{\pi}'\colon \overline{J}^{d,\phi'}_{g,n}\to\overline{\mathcal{M}}_{g,n}$, we have $\bar{\pi}'=\bar{\pi}\circ r$.

    Because $r$ is a morphism between proper relative algebraic spaces over $\overline{\mathcal{M}}_{g,n}$, $r$ itself is a proper morphism. 

    Since $r$ is an isomorphism over $\bar{\pi}'^{-1}(\mathcal{M}_{g,n})$, the image of $r$ contains the dense open subset $\bar{\pi}^{-1}(\mathcal{M}_{g,n})$. Since $r$ is proper, its image is closed; therefore, the image of $r$ must be $\overline{J}^{d,\phi}_{g,n}$.
\end{proof}
\begin{lemma}\label{lem:502}
    The intersection complex $\IC(\overline{J}^{d,\phi}_{g,n})$ is a direct summand of $r_*\mathbb{Q}_{\overline{J}^{d,\phi'}_{g,n}}[\dim \overline{J}^{d,\phi'}_{g,n}]$.
\end{lemma}
\begin{proof}
    Because $\phi'$ is a nondegenerate stability condition, the moduli space $\overline{J}^{d,\phi'}_{g,n}$ is a smooth Deligne--Mumford stack. Therefore, its intersection cohomology complex with rational coefficients is isomorphic to the shifted constant sheaf:
    $$ \IC(\overline{J}^{d,\phi'}_{g,n}) \simeq \mathbb{Q}_{\overline{J}^{d,\phi'}_{g,n}}[\dim \overline{J}^{d,\phi'}_{g,n}]. $$
    
    From the previous proposition, the map $r\colon\overline{J}^{d,\phi'}_{g,n}\to \overline{J}^{d,\phi}_{g,n}$ is a proper morphism. Furthermore, we know that $r$ restricts to an isomorphism of the dense open locus over $\mathcal{M}_{g,n}$. Therefore, applying the BBDG decomposition theorem, $\IC(\overline{J}^{d,\phi}_{g,n})$ must be a direct summand in the decomposition.
\end{proof}
We can provide an alternative proof of Theorem~\ref{1A}:
\begin{proof}[Second proof to Theorem~\ref{1A}, $n\geq1$ case]
    By \cite[Theorem 5.12]{miglioriniSupportTheoremHilbert2021}, each direct summand of 
    $\bar{\pi}'_*\mathbb{Q}_{\overline{J}^{d,\phi'}_{g,n}}[\dim \overline{J}^{d,\phi'}_{g,n}]$ is of full support.

    On the other hand, $\IC(\overline{J}^{d,\phi}_{g,n})$ is a direct summand of $r_*\mathbb{Q}_{\overline{J}^{d,\phi'}_{g,n}}[\dim \overline{J}^{d,\phi'}_{g,n}]$. Therefore, $\bar{\pi}_*\IC(\overline{J}^{d,\phi}_{g,n})$ is a direct summand of $\bar{\pi}'_*\mathbb{Q}_{\overline{J}^{d,\phi'}_{g,n}}[\dim \overline{J}^{d,\phi'}_{g,n}]$. Therefore, support of each direct summand of $\bar{\pi}_*\IC(\overline{J}^{d,\phi}_{g,n})$ must be the total space $\overline{\mathcal{M}}_{g,n}$. The rest of the proof is analogous to that demonstrated in Subsection \ref{sec:proof1}.
\end{proof}
\subsection{The \texorpdfstring{$n=0$}{n=0} case}
When $n=0$, there exist certain pairs of $(g,d)$ such that the elements in the universal stability space $V^{d}_{g,n}$ are always degenerate. More precisely, from \cite[Remark 5.12]{kassStabilitySpaceCompactified2019}, there exists a nondegenerate element if and only if 
$$\gcd(d+1-g,2g-2)=1.$$

While the monodromy of $\overline{\mathcal{M}}_{g}$ prevents the existence of a global nondegenerate perturbation of $\phi$, we can perturb the stability condition when we base change to an \'{e}tale neighborhood. 

Let $x = [C]\in \overline{\mathcal{M}}_{g}$ be a closed point corresponding to a nodal curve, with $m$ irreducible components 
$$C=\bigcup_{i=1,2,\dots,m} C_i.$$

    By \cite[Lemma 28]{estevesCompactifyingRelativeJacobian2001}, there exists an \'{e}tale morphism $p \colon U \to \overline{\mathcal{M}}_{g}$, containing $x$ in its image and sections
    $$\sigma_1,\dots,\sigma_m\colon U\to \mathcal{C}_U$$
of the pullback of the universal curve $\mathcal{C}_U\to U$, such that:
\begin{enumerate}
    \item for every $x'\in U$ and corresponding curve $C'$, every $\sigma_i(x')$ belongs to the smooth locus of $C'$;
    \item every irreducible component of $C'$ contains $\sigma_i(x')$ for some $i$.
\end{enumerate}

    By the existence of these sections, we can construct a perturbation of the degenerate stability condition $\phi$ over $U$.

\begin{proposition}\label{prop:localretract}
    Let $x = [C] \in \overline{\mathcal{M}}_{g}$ be a closed point, and let $\phi \in V^d_{g}$ be a degenerate stability condition. There exists an \'{e}tale neighborhood $p \colon U \to \overline{\mathcal{M}}_{g}$ of $x$, a local nondegenerate stability condition $\phi'_U$ over $U$, and a proper surjective morphism
    $$ r_U \colon \overline{J}^{d,\phi'_U}_U \to \overline{J}^{d,\phi\vert_U}_U,$$
    where $\overline{J}^{d,\phi\vert_U}_U \simeq \overline{J}^{d,\phi}_{g} \times_{\overline{\mathcal{M}}_{g}} U$.
\end{proposition}

\begin{proof}
We define a perturbed local stability condition over $U$ such that for any irreducible component $C'_j$ of the fiber of $x'=[C']$:
$$ \phi'_U(C'_j) = \phi(C'_j) + \sum_{i=1}^m \epsilon_i \delta_{ij},$$where the perturbation parameters $\epsilon_i$ are explicitly chosen as rational numbers $\epsilon_i \in \mathbb{Q}$ satisfying $\sum_{i=1}^m \epsilon_i=0$, and$$
\delta_{ij}=\begin{cases}
    1, & \text{if } \sigma_i(x')\in C'_j,\\
    0, & \text{if } \sigma_i(x')\notin C'_j.
\end{cases}
$$

This definition respects the contraction of the dual graph, and one can check it is a legal stability condition over $U$, corresponding to the $\mathbb{Q}$-line bundle $\mathcal{O}_{\mathcal{C}_{U}}(\sum_{i=1}^{m}\epsilon_i\sigma_i)$. 

To guarantee that $\phi'_U$ avoids all stability walls simultaneously over $U$, we consider the condition for a stability parameter to land on a wall. Following the setting in \cite{miglioriniSupportTheoremHilbert2021}, a stability wall corresponds to a proper subcurve $C'_0 \subset C'$ such that the sum of the stability conditions over its components is an integer: $\sum_{C'_j \subseteq C'_0} \phi'_U(C'_j) \in \mathbb{Z}$. 

Substituting the definition of $\phi'_U$, this condition requires:
$$\sum_{C'_j \subseteq C'_0} \phi(C'_j) + \sum_{i \in I(C'_0)} \epsilon_i \in \mathbb{Z},$$
where $I(C'_0)$ indexes sections $\sigma_i$ that land in $C'_0$. Because the \'etale neighborhood $U$ is stratified by only finitely many topological dual graphs, there are only finitely many subcurves $C'_0$ and corresponding subsets $I(C'_0)$ to consider. 

Thus, the equations $\sum_{i \in I(C'_0)} \epsilon_i \in\mathbb{Z}$ define a locally finite collection of hyperplanes within the parameter space $V^{0}_{U}\simeq\mathbb{Q}^{m-1}$. A rational vector $(\epsilon_1, \dots, \epsilon_m) \in V^{0}_{U}$ can always be chosen arbitrarily close to the origin such that it avoids this finite set of hyperplanes entirely. This ensures $\phi'_U$ is a nondegenerate stability condition, living in an adjacent chamber of $\phi\vert_U$.

Furthermore, any rank-1 torsion-free sheaf over $U$ that is $\phi'_U$-semistable is necessarily $\phi\vert_U$-semistable. Passing the inclusion map of moduli stacks to the good moduli space, we have a proper map:
    $$ r_U \colon \overline{J}^{d,\phi'_U}_U \longrightarrow \overline{J}^{d,\phi\vert_U}_U.$$
\end{proof}

\begin{proof}[Second proof to Theorem~\ref{1A}, $n=0$ case]
    Consider the decomposition
    $$\bar{\pi}_* \IC(\overline{J}^{d,\phi}_{g})\simeq \bigoplus_{\alpha,i}\IC (Z_\alpha, L_\alpha)[-i].$$

    If there is a direct summand such that its support $Z_\alpha\subsetneq\overline{\mathcal{M}}_{g}$ is a strict subset, choose a closed point $x\in Z_\alpha$. By Proposition~\ref{prop:localretract}, there exists an \'{e}tale neighborhood $p \colon U \to \overline{\mathcal{M}}_{g}$ of $x$, a nondegenerate stability condition $\phi'_U$ over $U$, and a contraction morphism:
    $$r_U \colon \overline{J}^{d,\phi'_U}_U \to \overline{J}^{d,\phi\vert_U}_U.$$

    Let $\pi'_U\colon \overline{J}^{d,\phi'_U}_U\to U$ be the forgetful map. We have $\pi'_U=\bar{\pi}\vert_U\circ r_U$. By \cite[Theorem 5.12]{miglioriniSupportTheoremHilbert2021}, the support of each direct summand of $\pi'_{U*}\mathbb{Q}_{\overline{J}^{d,\phi'_U}_U}$ must be $U$. On the other hand, $r_U$ induces an isomorphism on the loci over $p^{-1}(\mathcal{M}_g)\cap U$, therefore $\IC(\overline{J}^{d,\phi\vert_U}_U)$ is a direct summand of $r_{U*}(\mathbb{Q}_{\overline{J}^{d,\phi'_U}_U}[\dim \overline{J}^{d,\phi'_U}_U])$. It follows that $(\bar{\pi}\vert_U)_*\IC(\overline{J}^{d,\phi\vert_U}_U)$ is a direct summand of $\pi'_{U*}\mathbb{Q}_{\overline{J}^{d,\phi'_U}_U}$; therefore any direct summand of this pushed-forward intersection cohomology complex must have full support on $U$.

    However, $(\bar{\pi}\vert_U)_*\IC(\overline{J}^{d,\phi\vert_U}_U)=p^*(\bar{\pi}_* \IC(\overline{J}^{d,\phi}_{g}))$, so there is a direct summand supported on $p^{-1}(Z_\alpha)\cap U$. Therefore $U\subseteq p^{-1}(Z_\alpha)$, so $p(U)\subseteq Z_\alpha$. Since $Z_\alpha$ is closed, and $p(U)\subseteq \overline{\mathcal{M}}_{g}$ is open and dense, it follows that $Z_\alpha=\overline{\mathcal{M}}_{g}$.
\end{proof}

\printbibliography

\end{document}